\documentclass[11pt,oneside,reqno,american]{amsart}
\usepackage[T1]{fontenc}
\usepackage[latin9]{inputenc}
\usepackage{xcolor}
\usepackage{pdfcolmk}
\usepackage{verbatim}
\usepackage{amstext}
\usepackage{amsthm}
\usepackage{amssymb}
\usepackage{undertilde}
\usepackage{xargs}[2008/03/08]
\usepackage[all]{xy}
\PassOptionsToPackage{normalem}{ulem}
\usepackage{ulem}

\makeatletter

\providecolor{lyxadded}{rgb}{0,0,1}
\providecolor{lyxdeleted}{rgb}{1,0,0}

\DeclareRobustCommand{\lyxsout}[1]{\ifx\\#1\else\sout{#1}\fi}

\numberwithin{equation}{section}
\numberwithin{figure}{section}
\theoremstyle{plain}
\newtheorem{thm}{\protect\theoremname}[section]
\theoremstyle{plain}
\newtheorem{cor}[thm]{\protect\corollaryname}
\theoremstyle{remark}
\newtheorem{rem}[thm]{\protect\remarkname}

\usepackage{mathtools}
\usepackage{amsthm}
\usepackage{pdfsync} 
\usepackage{hyperref}
\usepackage[all]{xy}

 \theoremstyle{plain}
  \newtheorem{myprop}[thm]{Proposition}
\usepackage{fbb}
\usepackage[stix2,scaled=0.94]{newtxmath}
\usepackage[cal=boondoxo,bb=boondox,frak=boondox]{mathalfa}
\usepackage[scaled=.9]{rsfso}
\usepackage{xcolor} 
\definecolor{brown(traditional)}{rgb}{0.59, 0.29, 0.0}
\definecolor{blue(ryb)}{rgb}{0.01, 0.28, 1.0}
\definecolor{red}{rgb}{1.0, 0.0, 0.0}
\definecolor{magenta}{rgb}{1.0, 0.0, 1.0}
\definecolor{mahogany}{rgb}{0.75, 0.25, 0.0}
\definecolor{lavenderpurple}{rgb}{0.59, 0.48, 0.71}
\definecolor{olive}{rgb}{0.5, 0.5, 0.0}
\definecolor{brickred}{rgb}{0.8, 0.25, 0.33}
\definecolor{antiquefuchsia}{rgb}{0.57, 0.36, 0.51}
\definecolor{bole}{rgb}{0.47, 0.27, 0.23}
\definecolor{darkolivegreen}{rgb}{0.33, 0.42, 0.18}
\definecolor{deepjunglegreen}{rgb}{0.0, 0.29, 0.29}
\definecolor{brickred}{rgb}{0.8, 0.25, 0.33}
\definecolor{deepjunglegreen}{rgb}{0.0, 0.29, 0.29}
\definecolor{darkpastelgreen}{rgb}{0.01, 0.75, 0.24}
\definecolor{green(pigment)}{rgb}{0.0, 0.65, 0.31}
\definecolor{junglegreen}{rgb}{0.16, 0.67, 0.53}
\definecolor{officegreen}{rgb}{0.0, 0.5, 0.0}
\definecolor{seagreen}{rgb}{0.18, 0.55, 0.34}
\definecolor{teal}{rgb}{0.0, 0.5, 0.5}
\definecolor{brightgreen}{rgb}{0.4, 1.0, 0.0}
\definecolor{electricgreen}{rgb}{0.0, 1.0, 0.0}
\definecolor{malachite}{rgb}{0.04, 0.85, 0.32}

\usepackage{undertilde}
\usepackage{accents}

\makeatother

\usepackage{babel}
\providecommand{\corollaryname}{Corollary}
\providecommand{\remarkname}{Remark}
\providecommand{\theoremname}{Theorem}

\begin{document}

\global\long\def\ga{\alpha}%
\global\long\def\gb{\beta}%
\global\long\def\ggm{\gamma}%
\global\long\def\go{\omega}%
\global\long\def\gs{\sigma}%
\global\long\def\gd{\delta}%
\global\long\def\gD{\Delta}%
\global\long\def\vph{\phi}%
\global\long\def\gf{\varphi}%
\global\long\def\gk{\kappa}%
\global\long\def\gl{\lambda}%
\global\long\def\gz{\zeta}%
\global\long\def\gh{\eta}%
\global\long\def\gy{\upsilon}%
\global\long\def\gth{\theta}%
\global\long\def\gO{\Omega}%
\global\long\def\gG{\Gamma}%

\global\long\def\eps{\varepsilon}%
\global\long\def\epss#1#2{\varepsilon_{#2}^{#1}}%
\global\long\def\ep#1{\eps_{#1}}%

\global\long\def\wh#1{\widehat{#1}}%
\global\long\def\hi{\hat{\imath}}%
\global\long\def\hj{\hat{\jmath}}%
\global\long\def\hk{\hat{k}}%
\global\long\def\ol#1{\overline{#1}}%
\global\long\def\ul#1{\underline{#1}}%

\global\long\def\spec#1{\textsf{#1}}%

\global\long\def\ui{\wh{\boldsymbol{\imath}}}%
\global\long\def\uj{\wh{\boldsymbol{\jmath}}}%
\global\long\def\uk{\widehat{\boldsymbol{k}}}%

\global\long\def\uI{\widehat{\mathbf{I}}}%
\global\long\def\uJ{\widehat{\mathbf{J}}}%
\global\long\def\uK{\widehat{\mathbf{K}}}%

\global\long\def\mc#1{\mathcal{#1}}%
\global\long\def\bs#1{\boldsymbol{#1}}%
\global\long\def\vect#1{\mathbf{#1}}%
\global\long\def\bi#1{\textbf{\emph{#1}}}%

\global\long\def\uv#1{\widehat{\boldsymbol{#1}}}%
\global\long\def\cross{\times}%

\global\long\def\di{d}%
\global\long\def\dee#1{\mathop{d#1}}%

\global\long\def\ddt{\frac{\dee{}}{\dee t}}%
\global\long\def\dbyd#1{\frac{\dee{}}{\dee{#1}}}%
\global\long\def\dby#1#2{\frac{\partial#1}{\partial#2}}%
\global\long\def\dxdt#1{\frac{\dee{#1}}{\dee t}}%

\global\long\def\vct#1{\bs{#1}}%

\global\long\def\partialby#1#2{\frac{\partial#1}{\partial x^{#2}}}%
\newcommandx\parder[2][usedefault, addprefix=\global, 1=]{\frac{\partial#2}{\partial#1}}%
\global\long\def\supdot{^{\!\bs{\mathord{\cdot}}}}%

\global\long\def\fall{,\quad\text{for all}\quad}%

\global\long\def\reals{\mathbb{R}}%

\global\long\def\rthree{\reals^{3}}%
\global\long\def\rsix{\reals^{6}}%
\global\long\def\rn{\reals^{n}}%
\global\long\def\eucl{\mathbb{E}}%
\global\long\def\euthree{\eucl^{3}}%
\global\long\def\euln{\eucl^{n}}%

\global\long\def\prn{\reals^{n+}}%
\global\long\def\nrn{\reals^{n-}}%
\global\long\def\cprn{\overline{\reals}^{n+}}%
\global\long\def\cnrn{\overline{\reals}^{n-}}%
\global\long\def\rt#1{\reals^{#1}}%
\global\long\def\rtw{\reals^{12}}%

\global\long\def\les{\leqslant}%
\global\long\def\ges{\geqslant}%

\global\long\def\dX{\dee{\bp}}%
\global\long\def\dx{\dee x}%
\global\long\def\D{D}%

\global\long\def\from{\colon}%
\global\long\def\tto{\longrightarrow}%
\global\long\def\lmt{\longmapsto}%
\global\long\def\lhr{\lhook\joinrel\longrightarrow}%
\global\long\def\mto{\mapsto}%

\global\long\def\abs#1{\left|#1\right|}%

\global\long\def\isom{\cong}%

\global\long\def\comp{\circ}%

\global\long\def\cl#1{\overline{#1}}%

\global\long\def\fun{\varphi}%

\global\long\def\interior{\textrm{Int}\,}%
\global\long\def\inter#1{\kern0pt  #1^{\mathrm{o}}}%
\global\long\def\interior{\textrm{Int}\,}%
\global\long\def\inter#1{\kern0pt  #1^{\mathrm{o}}}%
\global\long\def\into{\mathrm{o}}%

\global\long\def\sign{\textrm{sign}\,}%
\global\long\def\sgn#1{(-1)^{#1}}%
\global\long\def\sgnp#1{(-1)^{\abs{#1}}}%

\global\long\def\du#1{#1^{*}}%

\global\long\def\tsum{{\textstyle \sum}}%
\global\long\def\lsum{{\textstyle \sum}}%

\global\long\def\dimension{\textrm{dim}\,}%

\global\long\def\esssup{\textrm{ess}\,\sup}%

\global\long\def\ess{\textrm{{ess}}}%

\global\long\def\kernel{\mathop{\textrm{\textup{Kernel}}}}%

\global\long\def\support{\mathop{\textrm{\textup{supp}}}}%

\global\long\def\image{\mathop{\textrm{\textup{Image}}}}%

\global\long\def\diver{\mathop{\textrm{\textup{div}}}}%

\global\long\def\spanv{\textrm{span}}%

\global\long\def\tr{\mathop{\textrm{\textup{tr}}}}%
\global\long\def\tran{\mathrm{tr}}%

\global\long\def\opt{\mathrm{opt}}%

\global\long\def\resto#1{|_{#1}}%
\global\long\def\incl{\mathcal{I}}%
\global\long\def\iden{\imath}%
\global\long\def\idnt{\textrm{Id}}%
\global\long\def\rest{\rho}%
\global\long\def\extnd{e_{0}}%

\global\long\def\proj{\textrm{pr}}%

\global\long\def\L#1{L\bigl(#1\bigr)}%
\global\long\def\LS#1{L_{S}\bigl(#1\bigr)}%

\global\long\def\ino#1{\int_{#1}}%

\global\long\def\half{\frac{1}{2}}%
\global\long\def\shalf{{\scriptstyle \half}}%
\global\long\def\third{\frac{1}{3}}%

\global\long\def\empt{\varnothing}%

\global\long\def\innp#1#2{\left\langle #1,#2\right\rangle }%

\global\long\def\paren#1{\left(#1\right)}%
\global\long\def\bigp#1{\bigl(#1\bigr)}%
\global\long\def\biggp#1{\biggl(#1\biggr)}%
\global\long\def\Bigp#1{\Bigl(#1\Bigr)}%

\global\long\def\braces#1{\left\{  #1\right\}  }%
\global\long\def\sqbr#1{\left[#1\right]}%
\global\long\def\anglep#1{\left\langle #1\right\rangle }%

\global\long\def\bigabs#1{\bigl|#1\bigr|}%
\global\long\def\dotp#1{#1^{\centerdot}}%
\global\long\def\pdot#1{#1^{\bs{\!\cdot}}}%

\global\long\def\eq{\sim}%
\global\long\def\quot{/\!\!\eq}%
\global\long\def\by{\!/\!}%

\global\long\def\stp{\text{\small\ensuremath{\bigodot}}}%
\global\long\def\tp{\text{\small\ensuremath{\bigotimes}}}%

\global\long\def\mi#1{#1}%
\global\long\def\mii{I}%
\global\long\def\mie#1#2{#1_{1}\cdots#1_{#2}}%

\global\long\def\smi#1{\boldsymbol{#1}}%
\global\long\def\asmi#1{#1}%
\global\long\def\ordr#1{\left\langle #1\right\rangle }%

\global\long\def\symm#1{\paren{#1}}%
\global\long\def\smtr{\mathcal{S}}%

\global\long\def\perm{p}%
\global\long\def\sperm{\mathcal{P}}%

\global\long\def\oneto{1,\dots,}%

\global\long\def\lisub#1#2#3{#1_{1}#2\dots#2#1_{#3}}%

\global\long\def\lisup#1#2#3{#1^{1}#2\dots#2#1^{#3}}%

\global\long\def\lisubb#1#2#3#4{#1_{#2}#3\dots#3#1_{#4}}%

\global\long\def\lisubbc#1#2#3#4{#1_{#2}#3\cdots#3#1_{#4}}%

\global\long\def\lisubbwout#1#2#3#4#5{#1_{#2}#3\dots#3\widehat{#1}_{#5}#3\dots#3#1_{#4}}%

\global\long\def\lisubc#1#2#3{#1_{1}#2\cdots#2#1_{#3}}%

\global\long\def\lisupc#1#2#3{#1^{1}#2\cdots#2#1^{#3}}%

\global\long\def\lisupp#1#2#3#4{#1^{#2}#3\dots#3#1^{#4}}%

\global\long\def\lisuppc#1#2#3#4{#1^{#2}#3\cdots#3#1^{#4}}%

\global\long\def\lisuppwout#1#2#3#4#5#6{#1^{#2}#3#4#3\wh{#1^{#6}}#3#4#3#1^{#5}}%

\global\long\def\lisubbwout#1#2#3#4#5#6{#1_{#2}#3#4#3\wh{#1}_{#6}#3#4#3#1_{#5}}%

\global\long\def\lisubwout#1#2#3#4{#1_{1}#2\dots#2\widehat{#1}_{#4}#2\dots#2#1_{#3}}%

\global\long\def\lisupwout#1#2#3#4{#1^{1}#2\dots#2\widehat{#1^{#4}}#2\dots#2#1^{#3}}%

\global\long\def\lisubwoutc#1#2#3#4{#1_{1}#2\cdots#2\widehat{#1}_{#4}#2\cdots#2#1_{#3}}%

\global\long\def\twp#1#2#3{\dee{#1}^{#2}\wedge\dee{#1}^{#3}}%

\global\long\def\thp#1#2#3#4{\dee{#1}^{#2}\wedge\dee{#1}^{#3}\wedge\dee{#1}^{#4}}%

\global\long\def\fop#1#2#3#4#5{\dee{#1}^{#2}\wedge\dee{#1}^{#3}\wedge\dee{#1}^{#4}\wedge\dee{#1}^{#5}}%

\global\long\def\idots#1{#1\dots#1}%
\global\long\def\icdots#1{#1\cdots#1}%

\global\long\def\norm#1{\|#1\|}%

\global\long\def\nonh{\heartsuit}%

\global\long\def\nhn#1{\norm{#1}^{\nonh}}%

\global\long\def\bigmid{\,\bigl|\,}%

\global\long\def\trps{^{{\scriptscriptstyle \textsf{T}}}}%

\global\long\def\testfuns{\mathcal{D}}%

\global\long\def\ntil#1{\tilde{#1}{}}%

\global\long\def\pis{y}%
\global\long\def\xo{\pis_{0}}%
\global\long\def\x{x}%

\global\long\def\pib{x}%
\global\long\def\bp{X}%
\global\long\def\ii{i}%
\global\long\def\ia{\alpha}%
\global\long\def\fp{y}%
\global\long\def\piv{v}%

\global\long\def\ib{i}%
\global\long\def\is{\alpha}%

\global\long\def\pbndo{\Gamma}%
\global\long\def\bndoo{\pbndo_{0}}%
 
\global\long\def\bndot{\pbndo_{t}}%
\global\long\def\intb{\inter{\body}}%
\global\long\def\bndb{\bdry\body}%

\global\long\def\cloo{\cl{\gO}}%

\global\long\def\nor{\nu}%
\global\long\def\Nor{\mathbf{N}}%

\global\long\def\dA{\dee A}%

\global\long\def\dV{\dee V}%

\global\long\def\eps{\varepsilon}%

\global\long\def\tv{v}%
\global\long\def\av{u}%

\global\long\def\svs{\mathcal{W}}%
\global\long\def\vs{\mathbf{V}}%
\global\long\def\avs{\mathbf{U}}%
\global\long\def\affsp{\mathcal{A}}%
\global\long\def\man{\mathcal{M}}%
\global\long\def\odman{\mathcal{N}}%
\global\long\def\subman{\mathcal{V}}%
\global\long\def\pt{p}%

\global\long\def\vbase{e}%
\global\long\def\sbase{\mathbf{e}}%
\global\long\def\msbase{\mathfrak{e}}%
\global\long\def\vect{v}%
\global\long\def\dbase{\sbase}%

\global\long\def\chart{\varphi}%
\global\long\def\Chart{\Phi}%

\global\long\def\mind{\alpha}%
\global\long\def\vb{W}%
\global\long\def\vbp{\pi}%

\global\long\def\vbt{\mathcal{E}}%
\global\long\def\fib{\vs}%
\global\long\def\vbts{W}%
\global\long\def\avb{U}%
\global\long\def\vbp{\xi}%

\global\long\def\chart{\vph}%
\global\long\def\vbchart{\Phi}%

\global\long\def\jetb#1{J^{#1}}%
\global\long\def\jet#1{j^{1}(#1)}%
\global\long\def\tjet{\tilde{\jmath}}%

\global\long\def\Jet#1{J^{1}(#1)}%

\global\long\def\jetm{j}%

\global\long\def\coj{\mathfrak{d}}%

\global\long\def\alt{\mathfrak{A}}%

\global\long\def\pou{\eta}%

\global\long\def\ext{{\textstyle \bigwedge}}%
\global\long\def\forms{\Omega}%

\global\long\def\dotwedge{\dot{\mbox{\ensuremath{\wedge}}}}%

\global\long\def\vel{\theta}%

\global\long\def\Jac{\mathcal{J}}%

\global\long\def\contr{\mathbin{\raisebox{0.4pt}{\mbox{\ensuremath{\lrcorner}}}}}%
\global\long\def\fcor{\llcorner}%
\global\long\def\bcor{\lrcorner}%
\global\long\def\fcontr{\mathbin{\raisebox{0.4pt}{\mbox{\ensuremath{\llcorner}}}}}%

\global\long\def\lie{\mathcal{L}}%

\global\long\def\ssym#1#2{\ext^{#1}T^{*}#2}%

\global\long\def\sh{^{\sharp}}%

\global\long\def\nfo{\ext^{n}T^{*}\base}%
\global\long\def\dfs{\ext^{d}T^{*}\base}%
\global\long\def\dmfs{\ext^{d-1}T^{*}\base}%

\global\long\def\spc{\mathcal{S}}%
\global\long\def\sptm{\mathcal{E}}%
\global\long\def\evnt{e}%
\global\long\def\frame{\Psi}%

\global\long\def\timeman{\mathcal{T}}%
\global\long\def\zman{t}%
\global\long\def\dims{n}%
\global\long\def\m{\dims-1}%
\global\long\def\dimw{m}%

\global\long\def\wc{z}%

\global\long\def\fourv#1{\mbox{\ensuremath{\mathfrak{#1}}}}%

\global\long\def\pbform#1{\undertilde{#1}}%
\global\long\def\util#1{\raisebox{-5pt}{\ensuremath{{\scriptscriptstyle \sim}}}\!\!\!#1}%

\global\long\def\utilJ{\util J}%

\global\long\def\utilRho{\util{\rho}}%

\global\long\def\body{\mathcal{B}}%
\global\long\def\man{\mathcal{M}}%
\global\long\def\var{\mathcal{V}}%
\global\long\def\base{\mathcal{X}}%
\global\long\def\fb{\mathcal{Y}}%
\global\long\def\srfc{\mathcal{Z}}%
\global\long\def\dimb{n}%
\global\long\def\dimf{m}%
\global\long\def\afb{\mathcal{Z}}%

\global\long\def\bdry{\partial}%

\global\long\def\gO{\varOmega}%

\global\long\def\reg{\gO}%
\global\long\def\bdrr{\bdry\reg}%

\global\long\def\bdom{\bdry\gO}%

\global\long\def\bndo{\partial\gO}%

\global\long\def\tpr{\vartheta}%

\global\long\def\mot{M}%
\global\long\def\vf{w}%
\global\long\def\const{h}%

\global\long\def\avf{u}%

\global\long\def\stn{\varepsilon}%
\global\long\def\djet{\chi}%

\global\long\def\jvf{\eps}%

\global\long\def\rig{r}%

\global\long\def\rigs{\mathcal{R}}%

\global\long\def\qrigs{\!/\!\rigs}%

\global\long\def\qd{\!/\,\!\kernel\diffop}%

\global\long\def\dis{\chi}%
\global\long\def\conf{\kappa}%
\global\long\def\invc{\hat{\conf}^{-1}}%
\global\long\def\dinvc{\hat{\conf}^{-1*}}%
\global\long\def\csp{\mathcal{Q}}%

\global\long\def\embds{\textrm{Emb}}%

\global\long\def\lc{A}%

\global\long\def\lv{\dot{A}}%
\global\long\def\alv{\dot{B}}%

\global\long\def\j{\mathop{\mathrm{j}}}%
\global\long\def\mapp{M}%
\global\long\def\J{J}%
\global\long\def\jex{\mathop{}\!\mathrm{j}}%

\global\long\def\fc{F}%
\global\long\def\load{f}%
\global\long\def\afc{g}%

\global\long\def\bfc{\mathbf{b}}%
\global\long\def\bfcc{b}%

\global\long\def\sfc{\mathbf{t}}%
\global\long\def\sfcc{t}%

\global\long\def\stm{\varsigma}%
\global\long\def\std{S}%
\global\long\def\tst{\sigma}%
\global\long\def\tstd{s}%
\global\long\def\st{\sigma}%
\global\long\def\vst{\varsigma}%
\global\long\def\vstd{S}%
\global\long\def\tstm{\sigma}%
\global\long\def\vstm{\varsigma}%

\global\long\def\stp{S_{P}}%
\global\long\def\slf{R}%

\global\long\def\crel{\Phi}%

\global\long\def\stmat{\tau}%

\global\long\def\gdiv{\bdry\textrm{iv\,}}%
\global\long\def\extjet{\mathfrak{d}}%

\global\long\def\smc#1{\mathfrak{#1}}%

\global\long\def\nhs{P}%
\global\long\def\nhsa{P}%
\global\long\def\nhsb{\underline{P}}%

\global\long\def\soc{Z}%

\global\long\def\sts{\varSigma}%
\global\long\def\spstd{\mathfrak{S}}%
\global\long\def\sptst{\mathfrak{T}}%
\global\long\def\spnhs{\mathcal{P}}%
\global\long\def\Ljj{\L{J^{1}(J^{k-1}\vb),\ext^{n}T^{*}\base}}%

\global\long\def\spsb{\text{\Large\ensuremath{\Delta}}}%

\global\long\def\ened{\mathfrak{w}}%
\global\long\def\energy{\mathfrak{W}}%

\global\long\def\ebdfc{T}%
\global\long\def\optimum{\st^{\textrm{opt}}}%
\global\long\def\scf{K}%

\global\long\def\grp{G}%
\global\long\def\gact{A}%
\global\long\def\gid{e}%
\global\long\def\gel{\ggm}%

\global\long\def\ael{\upsilon}%
\global\long\def\lal{\mathfrak{g}}%

\global\long\def\prop{P}%
\global\long\def\expr{\Pi}%

\global\long\def\aprop{Q}%

\global\long\def\flux{\omega}%
\global\long\def\aflux{\psi}%

\global\long\def\fform{\tau}%

\global\long\def\dimn{n}%

\global\long\def\sdim{{\dimn-1}}%

\global\long\def\fdens{\phi}%

\global\long\def\pform{s}%
\global\long\def\vform{\beta}%
\global\long\def\sform{\tau}%
\global\long\def\flow{\vf}%
\global\long\def\n{\m}%
\global\long\def\cmap{\mathfrak{t}}%
\global\long\def\vcmap{\varSigma}%

\global\long\def\mvec{\mathfrak{v}}%
\global\long\def\mveco#1{\mathfrak{#1}}%
\global\long\def\mv#1{\mathfrak{#1}}%
\global\long\def\smbase{\mathfrak{e}}%
\global\long\def\spx{\simp}%
\global\long\def\il{l}%
\global\long\def\awe{\frown}%

\global\long\def\hp{H}%
\global\long\def\ohp{h}%

\global\long\def\hps{G_{\dims-1}(T\spc)}%
\global\long\def\ohps{G_{\dims-1}^{\perp}(T\spc)}%

\global\long\def\hyper{\mathcal{S}}%

\global\long\def\hpsx{G_{\dims-1}(\tspc)}%
\global\long\def\ohpsx{G_{\dims-1}^{\perp}(\tspc)}%

\global\long\def\fbun{F}%

\global\long\def\flowm{\Phi}%

\global\long\def\tgb{T\spc}%
\global\long\def\ctgb{T^{*}\spc}%
\global\long\def\tspc{T_{\pis}\spc}%
\global\long\def\dspc{T_{\pis}^{*}\spc}%

\global\long\def\fflow{\fourv J}%
\global\long\def\fvform{\mathfrak{b}}%
\global\long\def\fsform{\mathfrak{t}}%
\global\long\def\fpform{\mathfrak{s}}%
\global\long\def\lfc{\mathfrak{F}}%

\global\long\def\maxw{\mathfrak{g}}%
\global\long\def\frdy{\mathfrak{f}}%
\global\long\def\ptnl{\psi}%
\global\long\def\pts{\Psi}%
\global\long\def\tptn{\Psi}%
\global\long\def\vptn{\mathfrak{a}}%
\global\long\def\mtst{\tstd_{M}}%
\global\long\def\mvst{\vstd_{M}}%

\global\long\def\sobp#1#2{W_{#2}^{#1}}%

\global\long\def\inner#1#2{\left\langle #1,#2\right\rangle }%

\global\long\def\fields{\sobp pk(\vb)}%

\global\long\def\bodyfields{\sobp p{k_{\partial}}(\vb)}%

\global\long\def\forces{\sobp pk(\vb)^{*}}%

\global\long\def\bfields{\sobp p{k_{\partial}}(\vb\resto{\bndo})}%

\global\long\def\loadp{(\sfc,\bfc)}%

\global\long\def\strains{\lp p(\jetb k(\vb))}%

\global\long\def\stresses{\lp{p'}(\jetb k(\vb)^{*})}%

\global\long\def\diffop{D}%

\global\long\def\strainm{E}%

\global\long\def\incomps{\vbts_{\yieldf}}%

\global\long\def\devs{L^{p'}(\eta_{1}^{*})}%

\global\long\def\incompsns{L^{p}(\eta_{1})}%

\global\long\def\testf{\mathcal{D}}%
\global\long\def\dists{\mathcal{D}'}%

\global\long\def\codiv{\boldsymbol{\partial}}%

\global\long\def\currof#1{\tilde{#1}}%

\global\long\def\chn{c}%
\global\long\def\chnsp{\mathbf{C}}%

\global\long\def\current{T}%
\global\long\def\curr{R}%

\global\long\def\curd{S}%
\global\long\def\curwd#1{\wh{#1}}%
\global\long\def\curnd#1{\wh{#1}}%

\global\long\def\contrf{{\scriptstyle \smallfrown}}%

\global\long\def\prodf{{\scriptstyle \smallsmile}}%

\global\long\def\form{\omega}%

\global\long\def\dens{\rho}%

\global\long\def\simp{s}%
\global\long\def\ssimp{\Delta}%
\global\long\def\cpx{K}%

\global\long\def\cell{C}%

\global\long\def\chain{B}%
\global\long\def\A{A}%
\global\long\def\B{B}%

\global\long\def\ach{A}%

\global\long\def\coch{X}%

\global\long\def\scale{s}%

\global\long\def\fnorm#1{\norm{#1}^{\flat}}%

\global\long\def\chains{\mathcal{A}}%

\global\long\def\ivs{\boldsymbol{U}}%

\global\long\def\mvs{\boldsymbol{V}}%

\global\long\def\cvs{\boldsymbol{W}}%

\global\long\def\ndual#1{#1'}%

\global\long\def\nd{'}%

\global\long\def\cee#1{C^{#1}}%

\global\long\def\lone{\{L^{1}\}}%

\global\long\def\linf{L^{\infty}}%

\global\long\def\lp#1{L^{#1}}%

\global\long\def\ofbdo{(\bndo)}%

\global\long\def\ofclo{(\cloo)}%

\global\long\def\vono{(\gO,\rthree)}%

\global\long\def\lomu{\{L^{1,\mu}\}}%
\global\long\def\limu{L^{\infty,\mu}}%
\global\long\def\limub{\limu(\body,\rthree)}%
\global\long\def\lomub{\lomu(\body,\rthree)}%

\global\long\def\vonbdo{(\bndo,\rthree)}%
\global\long\def\vonbdoo{(\bndoo,\rthree)}%
\global\long\def\vonbdot{(\bndot,\rthree)}%

\global\long\def\vonclo{(\cl{\gO},\rthree)}%

\global\long\def\strono{(\gO,\reals^{6})}%

\global\long\def\sob{\{W_{1}^{1}\}}%

\global\long\def\sobb{\sob(\gO,\rthree)}%

\global\long\def\lob{\lone(\gO,\rthree)}%

\global\long\def\lib{\linf(\gO,\reals^{12})}%

\global\long\def\ofO{(\gO)}%

\global\long\def\oneo{{1,\gO}}%
\global\long\def\onebdo{{1,\bndo}}%
\global\long\def\info{{\infty,\gO}}%

\global\long\def\infclo{{\infty,\cloo}}%

\global\long\def\infbdo{{\infty,\bndo}}%
\global\long\def\lobdry{\lone(\bdry\gO,\rthree)}%

\global\long\def\ld{LD}%

\global\long\def\ldo{\ld\ofO}%
\global\long\def\ldoo{\ldo_{0}}%

\global\long\def\trace{\gamma}%
\global\long\def\dtrace{\delta}%
\global\long\def\gtrace{\beta}%

\global\long\def\pr{\proj_{\rigs}}%

\global\long\def\pq{\proj}%

\global\long\def\qr{\,/\,\reals}%

\global\long\def\aro{S_{1}}%
\global\long\def\art{S_{2}}%

\global\long\def\mo{m_{1}}%
\global\long\def\mt{m_{2}}%

\global\long\def\ebdfc{T}%

\global\long\def\mini{\Omega}%
\global\long\def\optimum{s^{\mathrm{opt}}}%
\global\long\def\scf{K}%
\global\long\def\opsf{\st^{\mathrm{opt}}}%
\global\long\def\doptimum{s^{\opt,{\scriptscriptstyle D}}}%
\global\long\def\loptimum{s^{\opt,{\scriptscriptstyle \mathcal{M}}}}%

\global\long\def\fsubs{M}%

\global\long\def\yieldc{B}%

\global\long\def\yieldf{Y}%

\global\long\def\trpr{\pi_{P}}%

\global\long\def\devpr{\pi_{\devsp}}%

\global\long\def\prsp{P}%

\global\long\def\devsp{D}%

\global\long\def\ynorm#1{\|#1\|_{\yieldf}}%

\global\long\def\colls{\Psi}%

\global\long\def\aro{S_{1}}%
\global\long\def\art{S_{2}}%

\global\long\def\mo{m_{1}}%
\global\long\def\mt{m_{2}}%

\global\long\def\trps{^{\mathsf{T}}}%

\global\long\def\hb{^{\mathrm{hb}}}%

\global\long\def\yieldst{s_{Y}}%

\global\long\def\yieldc{B}%

\global\long\def\lcap{C}%

\global\long\def\yieldf{Y}%

\global\long\def\sphpr{\pi_{P}}%

\global\long\def\devpr{\pi_{\devsp}}%

\global\long\def\prsp{P}%

\global\long\def\devsp{D}%

\global\long\def\ynorm#1{\|#1\|_{\yieldf}}%

\global\long\def\colls{\Psi}%

\global\long\def\cone{Q}%
\global\long\def\fpr{\Pi}%
\global\long\def\fprd{\fpr_{\devsp}}%
\global\long\def\fprp{\fpr_{\prsp}}%
\global\long\def\find{I_{\devsp}}%
\global\long\def\finp{I_{\prsp}}%
\global\long\def\fnorm#1{\norm{#1}_{\devsp}}%

\global\long\def\rig{r}%
\global\long\def\rigs{\mathcal{R}}%
\global\long\def\qrigs{\!/\!\rigs}%
\global\long\def\anv{\omega}%
\global\long\def\I{I}%
\global\long\def\mone{M_{1}}%

\global\long\def\bd{BD}%

\global\long\def\po{\proj_{0}}%
\global\long\def\normp#1{\norm{#1}'_{\ld}}%

\global\long\def\ssx{S}%

\global\long\def\smap{s}%

\global\long\def\smat{\chi}%

\global\long\def\sx{e}%

\global\long\def\snode{P}%
\global\long\def\newmacroname{}%

\global\long\def\elem{e}%

\global\long\def\nel{L}%

\global\long\def\el{l}%

\global\long\def\gr{g}%
\global\long\def\ngr{G}%

\global\long\def\eldof{\alpha}%

\global\long\def\glbs{\psi}%

\global\long\def\ipln{\phi}%

\global\long\def\ndof{D}%

\global\long\def\dof{d}%

\global\long\def\nldof{N}%

\global\long\def\ldof{n}%

\global\long\def\lvf{\chi}%

\global\long\def\amat{A}%
\global\long\def\bmat{B}%

\global\long\def\subsp{\mathcal{M}}%
\global\long\def\zerofn{Z}%

\global\long\def\snomat{E}%

\global\long\def\femat{E}%

\global\long\def\tmat{T}%

\global\long\def\fvec{f}%

\global\long\def\snsp{\mathcal{S}}%

\global\long\def\slnsp{\Phi}%
\global\long\def\dslnsp{\Phi^{{\scriptscriptstyle D}}}%

\global\long\def\ro{r_{1}}%

\global\long\def\rtwo{r_{2}}%

\global\long\def\rth{r_{3}}%

\global\long\def\fmax{M}%

\global\long\def\dform{\psi}%

\global\long\def\srfc{\mathcal{S}}%

\global\long\def\semib{\mathrm{SB}}%

\global\long\def\tm#1{\overrightarrow{#1}}%
\global\long\def\tmm#1{\underrightarrow{\overrightarrow{#1}}}%

\global\long\def\itm#1{\overleftarrow{#1}}%
\global\long\def\itmm#1{\underleftarrow{\overleftarrow{#1}}}%

\global\long\def\ptrac{\mathcal{P}}%

\global\long\def\nh#1{\hat{#1}}%
\global\long\def\nj{\hat{\jmath}}%
\global\long\def\nJ{\hat{J}}%
\global\long\def\rin#1{\mathfrak{#1}}%
\global\long\def\npi{\hat{\pi}}%
\global\long\def\rp{\rin p}%
\global\long\def\rq{\rin q}%
\global\long\def\rr{\rin r}%

\global\long\def\xty{(\base,\fb)}%
\global\long\def\xts{(\base,\spc)}%
\global\long\def\r{r}%
\global\long\def\ntm{(\reals^{n},\reals^{m})}%

\global\long\def\tproj{\frame_{\timeman}}%
\global\long\def\sproj{\frame_{\spc}}%

\global\long\def\cons{c}%
\global\long\def\optm{\go}%
\global\long\def\flxs{\mathcal{W}}%
\global\long\def\cost{Q}%

\global\long\def\mtn{e}%
\global\long\def\sppp{\lambda}%

\global\long\def\mtsp{\mathscr{E}}%

\global\long\def\disp{g}%
\global\long\def\diffs{G}%

\global\long\def\bv{BV}%

\title[Optimal Flux Fields]{Notes on Optimal Flux Fields}
\author{Vladimir Gol'dshtein$\vphantom{N^{2}}^{1}$  and Reuven Segev$\vphantom{N^{2}}^{2}$}
\address{}
\begin{abstract}
For a given region, and specified boundary flux and density rate of
an extensive property, the optimal flux field that satisfies the balance
conditions is considered. The optimization criteria are the $L^{p}$-norm
and a Sobolev-like norm of the flux field. Finally, the capacity of
the region to accommodate various boundary fluxes and density rates
is defined and analyzed.
\end{abstract}

\date{\today\\[2mm]
$^1$ Department of Mathematics, Ben-Gurion University of the Negev, Israel. Email: vladimir@bgu.ac.il\\
$^2$ Department of Mechanical Engineering, Ben-Gurion University of the Negev, Israel. Email: rsegev@post.bgu.ac.il}

\maketitle

\section{Introduction}

Consider the flux, or flow field, $\vf$, of an extensive property
in a given regular region, $\gO$, in a 3-dimensional Euclidean space.
The flux field, $\vf$, balances the rate of change, $\vform$, of
the density, $\dens$, of the property, as well as the flux density
on the boundary, $\sform$. This is classically expressed by the balance
conditions
\begin{equation}
\vform+\vf_{i,i}=0,\text{ in }\gO,\qquad\vf_{i}\nor_{i}=\sform,\text{ on }\bdry\gO,\label{eq:balanceC}
\end{equation}
where $\nor$ is the outwards pointing unit normal to the boundary.

Given the density rate and the flux density on the boundary, the balance
conditions above do not determine a unique flux field in $\gO$; we
have only one differential equation and one boundary condition for
the three components of the flux field. We use $\flxs_{\vform,\sform}$
to denote the set of flux fields that satisfy conditions (\ref{eq:balanceC})
for given $\vform$ and $\sform$.

Let 
\begin{equation}
\cost:\flxs_{\vform,\sform}\tto\reals
\end{equation}
be a cost function one wishes to minimize. That is, one considers
\begin{equation}
\optm_{\vform,\sform}=\inf\left\{ \cost(\vf)\mid\vf\in\flxs_{\vform,\sform}\right\} ,
\end{equation}
and studies the question of existence of minimizers.

For example, we envisage a continuous model of the traffic in a city
so that $\vf$ is the traffic density defined in the city, $\gO$.
At any particular time, there may be an estimated flux density of
traffic entering or leaving the city, as well as the required rate
of density of vehicles at the various locations in the city. Thus,
for safety reasons, one would like to find the velocity field of the
vehicles that will minimize the maximum speed taken over the city.
Concretely, one would like to minimize the $L^{\infty}$-norm of the
flux field, or generally, take
\begin{equation}
\cost(\vf)=\norm w_{L^{p}(\gO)},\qquad1\les p\les\infty.\label{eq:Lp}
\end{equation}

As another example, for a Newtonian fluid flow, one may wish to minimize
the energy dissipation, where $Q$ is given simply by 
\begin{equation}
\cost(\vf):=\int_{\gO}\vf_{i,j}\vf_{i,j}\dee V\label{eq:W12}
\end{equation}
(\emph{cf.}~\cite[p.~579--580]{Lamb32}, \cite[p.~152--153]{Batchelor67}).
An analogous cost function is suggested by the attempts to rationalize
the Kleiber rule for the geometry of animal organs (see \cite{PNAS_2010}).
For example, the geometric structure of the bronchial airways may
induce a flux field that minimizes energy dissipation during breathing.

In view of these observations, this article presents some notes concerning
the problem of optimizing flux fields (cf. \cite[Section 24.2.5]{Segev_Book_2023}).

In Section \ref{sec:SmoothExtensive}, considering smooth fields,
we review the basic fields and balance conditions corresponding to
a generic extensive property for which a flux vector field exists.
In particular, we present the functional forms of the balance conditions,
where the various fields act on real valued functions, $\ptnl$. Such
a test function, $\ptnl$, is interpreted physically as a potential
function for the extensive property, or as a resource density for
the ecological interpretation.

Section \ref{sec:Classical} presents a classical calculus of variation
analysis for the case where
\begin{equation}
\cost(\vf):=\int_{\gO}L(w_{i},w_{j,k})\dee V,
\end{equation}
for some real valued Lagrangian function $L:\reals^{12}\to\reals$.
It is not surprising that the analysis is analogous to that pertaining
to hyperelastic materials. 

The general construction for the analysis of optimization problems,
as in (\ref{eq:Lp}) and (\ref{eq:W12}), is outlined roughly in Section
\ref{sec:Weak-Formulation}. Trying to omit the technical details
as much as possible, no justification is given as to why the operations
performed are valid. The balance conditions are presented as an equality
of two bounded linear functionals on the space of potential functions.

Some basic properties of Sobolev spaces are reviewed briefly in Section
\ref{sec:Analytic-Preliminaries}. These properties are used later
in the analysis of the optimization problems corresponding to (\ref{eq:Lp})
and (\ref{eq:W12}).

The optimization of the $L^{p}$-norm of the flux field, as in (\ref{eq:Lp}),
is based on taking the space of potential function to be the Sobolev
space $\dot{W}^{1,q}(\gO)$, containing functions on $\gO$ having
zero average, and the partial derivatives of which are $L^{q}$, for
$q=p/(p-1)$. The analysis in Section \ref{sec:Lq-opt} proves the
existence of a minimizing flux field $\vf^{\opt}\in L^{p}(\gO)^{3}$.
Moreover, it is shown that 
\begin{equation}
\go_{(\vform,\sform)}=\norm{\vf^{\opt}}_{L^{p}(\gO)^{3}}=\sup\left\{ \frac{\int_{\gO}\vform\ptnl\dee{V+\int_{\bdry\gO}\sform\ptnl\dee A}}{\left[\int_{\gO}\abs{\nabla\ptnl}^{q}\dee V\right]^{1/q}}\mid\ptnl\in C^{\infty}(\cl{\gO})\right\} .
\end{equation}

For the minimization of (\ref{eq:W12}), we use the same general construction,
but here, the space of potential functions is taken as $\dot{W}^{2,2}(\gO)$.
This is space of equivalence classes of functions on $\gO$ under
the equivalence relation setting $\ptnl\sim\ptnl'$ if both functions
have the same second derivatives. In addition, it is required that
the second distributional derivatives of elements of $\dot{W}^{2,2}(\gO)$
be $L^{2}$. Here, a minimizer not only exists, but it is also unique.
The optimal value is given by
\begin{equation}
\go_{(\vform,\sform)}=\sup\left\{ \frac{\int_{\gO}\ptnl\dee{\vform+\int_{\bdry\gO}\ptnl\dee{\sform}}}{\left[\int_{\gO}\ptnl_{,ij}\ptnl_{,ij}\dee V\right]^{1/2}}\mid\ptnl\in C^{\infty}(\cl{\gO})\right\} .
\end{equation}

It is noted that in both optimization problems above, the data $\vform$
and $\sform$ may be irregular. For problem (\ref{eq:W12}), they
may be Radon measures on $\gO$ and its boundary, respectively.

The analysis of Section \ref{sec:Capacity} is relevant to the following
situation. Imagine that a control system navigates the flux field
so that it is optimal for given data $\vform$ and $\sform$. However,
the admissible flux fields are bounded by the condition $\norm{\vf}\les M$,
for a given bound $M$, reflecting the limitations of the system.
It is shown that there is a minimal number $C$, referred to as the
capacity of the region $\gO$, such that if $\norm{(\vform,\sform)}\les CM$,
the optimal flux field corresponding to $(\vform,\sform)$ is admissible\textemdash its
norm is not greater than $M$. The expression for the capacity is
shown to be related to the trace mapping of vector fields onto $\bdry\gO$.

\section{\label{sec:SmoothExtensive}Smoothly distributed extensive properties
and fluxes}

We consider the balance law for an extensive property, $\expr$, in
the physical space $\spc$. In an ecological context, one may conceive
the property, $\expr$, as a continuous distribution of a population
in the region $\gO\subset\spc$. In the analysis below, $\spc$ is
identified with $\rthree$, and $\gO$ is assumed to be a bounded
open subset with a smooth boundary. Most of what follows holds in
the more general case where the boundary, $\bdry\gO$, is only Lipschitz
continuous.

\subsection{Basic fields}

The density of the extensive property $\expr$ is given as a smooth
function $\dens$ defined in $\rthree$. Thus, for the region, $\gO$,
the total amount of the property in $\reg$, at any time, $t\in\reals$,
is given by
\begin{equation}
\expr(\reg)=\ino{\reg}\dens\dee V.
\end{equation}
In what follows, except for the region $\reg$ itself, all quantities
associated with the property $\expr$ will be time-dependent, although
we do not indicate this in the notation.

For the rate of change,
\begin{equation}
\frac{\dee{\expr(\reg)}}{\dee t}=\ino{\reg}\vform\dee V,\qquad\text{where,}\qquad\vform:=\frac{\dee{\dens}}{\dee t}.
\end{equation}

It is usually assumed in continuum mechanics that the time-derivative
$\dee{\expr}(\reg)/\dee t$ is a result of a source (or a sink) density
of the property in space and the total flux, $\Phi$, of the property,
which exits $\reg$ through the boundary, $\bdry\reg$. The \emph{source
density} of the property is given by a smooth function, $\pform$,
 so that the total production of the property in $\reg$ is given
by 
\begin{equation}
\ino{\reg}\pform\dV.
\end{equation}

The total flux on $\bdom$ is assumed to be given in terms of a smooth
function $\sform_{\gO}$ on $\bdry\reg$\textemdash the \emph{flux
density}\textemdash in the form
\begin{equation}
\Phi=\ino{\partial\reg}\sform_{\gO}\dA.\label{eq:Tot_Flux_and_Density}
\end{equation}
Thus, the classical balance law assumes the form 
\begin{equation}
\ino{\reg}\vform\dee V=\ino{\reg}\pform\dee V-\ino{\partial\reg}\sform_{\gO}\dee A.\label{eq:BalanceOnRegions}
\end{equation}

\subsection{The flux vector field}

If the boundary flux density, $\sform_{\mc{\mc R}}$, for the boundary
of any sub-region, $\mc R$, of $\gO$ satisfies the Cauchy postulates,
there is a flux vector field, $\flow$, defined in $\reg$, such that
for each sub-region, $\mc R$,
\begin{equation}
\sform_{\mc R}=\flow\cdot\nor_{\mc R},
\end{equation}
where $\nor_{\mc R}$ is the unit normal vector to $\bdry\mc R$.

Thus, Equation (\ref{eq:BalanceOnRegions}) is rewritten as
\begin{equation}
\ino{\mc R}\vform\dee V=\ino{\mc R}\pform\dee V-\ino{\partial\mc R}\flow\cdot\nor_{\mc R}\dee A.
\end{equation}
Using the Gauss theorem, and noting that the equation holds for every
sub-region $\mc R$, one obtains 
\begin{equation}
\vform+\nabla\cdot\flow=\pform.\label{eq:Balance_Smooth}
\end{equation}
When one views the last equation as a differential equation for the
flux field $\flow$ for given $\vform$ and $\pform$, supplemented
by the boundary condition 
\begin{equation}
\sform_{\gO}=\flow\cdot\nor_{\gO}\qquad\text{that we write for short as}\qquad\sform=\flow\cdot\nor,\label{eq:B-C}
\end{equation}
one notes that the system is under-determined. There is a class, $\mc B_{\gO}$,
of flux vector fields that satisfy these conditions.

\subsection{\label{subsec:functional-form}The functional form of the balance
equation}

Let $\ptnl$ be any smooth function defined on $\cl{\gO}$, which
we can interpret as a potential for the property $\expr$. Multiplying
the balance equation (\ref{eq:Balance_Smooth}) by $\ptnl$ and integrating
over $\gO$, one has,
\begin{equation}
\int_{\gO}\ptnl\vform\dee V+\int_{\gO}\ptnl\,\nabla\cdot\flow\dee V=\int_{\gO}\ptnl\pform\dee V.
\end{equation}
Using integration by parts, the Gauss theorem, and Equation (\ref{eq:B-C}),
the last equation is rewritten as 
\begin{equation}
\int_{\gO}\ptnl\vform\dee V+\int_{\bdry\gO}\ptnl\,\sform\dee V=\int_{\gO}\ptnl\pform\dee V+\int_{\gO}\nabla\ptnl\cdot\flow\dee V,
\end{equation}
or using components,
\begin{equation}
\int_{\gO}\ptnl\vform\dee V+\int_{\bdry\gO}\ptnl\,\sform\dee V=\int_{\gO}\ptnl\pform\dee V+\int_{\gO}\ptnl_{,i}\cdot\flow_{i}\dee V.\label{eq:weak-balance}
\end{equation}
The first integral on the left is interpreted as the rate of change
of the energy associated with the property $\expr$ in the region
$\gO$; the second integral on the left is interpreted as the energy
flux out of the region due to the flow $\flow$; the first integral
on the right is interpreted as the energy source in $\gO$ due to
the source of $\expr$; and the second term on the right is interpreted
as the change of total energy due to the flow of the property along
the potential gradient.

\section{\label{sec:Classical}The Classical Variational Analysis}

Let $\gO$ be a domain in $\rthree$ with a regular enough boundary
so that the Gauss theorem applies. Let $\vf$ be a flux field defined
in $\gO$ that satisfies the boundary conditions 
\begin{equation}
\vf\cdot\nor=\sform,\label{eq:bdry_flux}
\end{equation}
where, $\sform$ is a given flux distribution on the boundary and
$\nor$ is the unit outwards pointing normal. Assume, in addition,
that the flow field has to be compatible with a prescribed continuous
distribution $\vform$ of the rate of change of the density. In addition,
we assume that the property is conserved, so that $\pform=0$. Hence,
\begin{equation}
\vform+\vf_{i,i}=0.\label{eq:sinks}
\end{equation}

A classical optimization problem for the flux is given by a functional
\begin{equation}
\cost(\vf)=\int_{\gO}L(w_{i},w_{j,k})\dee V,\label{eq:functional}
\end{equation}
for a given smooth function 
\begin{equation}
L:\reals^{12}\tto\reals.
\end{equation}
For example, one may consider the minimization of the dissipation
power for which 
\begin{equation}
L(w_{i},w_{j,k})=\shalf\vf_{j,k}\vf_{j,k}.\label{eq:dissipation}
\end{equation}

Thus, for an optimal flux field one has to minimize $\cost(\vf)$
as given by (\ref{eq:functional}) subject to the constraint (\ref{eq:sinks})
and the constraint (\ref{eq:bdry_flux}) on the boundary, for given
functions $\vform$ and $\sform$ defined in $\gO$ and its boundary,
$\bdom$, respectively.

To account for the constraint (\ref{eq:functional}) in the variational
computation, we use a Lagrange multiplier function $\gl:\gO\tto\reals$
and consider the augmented functional
\begin{equation}
J:=\int_{\gO}[L(w_{i},w_{j,k})+\gl(\vform+\vf_{i,i})]\dee V.
\end{equation}
A necessary condition for $J$ to attain an extremal value at $\vf$
is 
\begin{equation}
DJ_{\vf}(u)=\parder[\eps]{}\left.\int_{\gO}[L(w_{i}+\eps u_{i},w_{i,j}+\eps u_{i,j})+\gl(\vform+\vf_{i,i}+\eps u_{i,i})]\dV\right|_{\eps=0}=0,
\end{equation}
for all variations $u$ that are compatible with the boundary constraint
(\ref{eq:bdry_flux}). Specifically 
\begin{equation}
(\vf+\eps u)\cdot\nor=\vf\cdot\nor=\sform,
\end{equation}
so we conclude that 
\begin{equation}
u\cdot\nor=0,
\end{equation}
on $\bdom$.

Carrying out the differentiation gives
\begin{equation}
\begin{split}0 & =\int_{\gO}\left[\frac{\bdry L}{\bdry\vf_{i}}u_{i}+\frac{\bdry L}{\bdry\vf_{i,j}}u_{i,j}+\gl u_{j,j}\right]\dee V,\\
 & =\int_{\gO}\left[\frac{\bdry L}{\bdry\vf_{i}}-\left(\frac{\bdry L}{\bdry\vf_{i,j}}\right)_{,j}-\gl_{,i}\right]u_{i}\dee V+\int_{\bdom}\left[\frac{\bdry L}{\bdry\vf_{i,j}}u_{i}\nor_{j}+\gl u_{j}\nor_{j}\right]\dee A,
\end{split}
\end{equation}
where it is noted that the second term in the integral over the boundary
vanishes.

Setting
\begin{equation}
\bfc_{i}:=-\frac{\bdry L}{\bdry\vf_{i}},\qquad\st_{ij}:=\frac{\bdry L}{\bdry\vf_{i,j}},\qquad\sfc_{i}=\st_{ij}\nor_{j},
\end{equation}
we have the conditions
\begin{equation}
\st_{ij,j}+\gl_{,i}+\bfc_{i}=0,\quad\text{in }\gO,\qquad\sfc_{i}u_{i}=0,\quad\text{on }\bdom.
\end{equation}
As $u$ is tangent to the boundary, we conclude that $\sfc$ must
be normal to the boundary.

For the particular case where $L$ is given by (\ref{eq:dissipation}),
we have $\bfc_{i}=0$, $\st_{ij}=\vf_{i,j}$ and one obtains the conditions
\begin{equation}
\vf_{i,jj}+\gl_{,i}=0,\quad\text{in }\gO,\qquad\vf_{i,j}n_{j}u_{i}=0,\quad\text{on }\bdom.
\end{equation}
Hence, on the boundary, the normal derivative of $\vf$ is perpendicular
to the boundary.

\section{The Weak Formulation of Balance\label{sec:Weak-Formulation}}

In this section, we roughly outline the structure needed for a weak
formulation of the balance law. The functional form of balance, as
presented in Equation (\ref{eq:weak-balance}), suggests that we view
the balance equation as an equality of two linear functionals acting
on an appropriate Banach space $\Psi$ of potential functions, $\psi$,
defined on $\gO$. The subset $\gO\subset\rthree$ is assumed to be
open and bounded, and the space $\Psi$ is assumed to contain the
restrictions of differentiable functions defined on $\cl{\gO}$. The
operations outlined in this general formal description are valid for
$C^{1}(\cl{\gO})$, and it is further assumed that these operations
may be extended to the entire space $\Psi$. Examples of such spaces
are given in Sections \ref{sec:Lq-opt} and \ref{sec:Dissipation-Optimization-in-W12}
below. 

We view the left-hand side of (\ref{eq:weak-balance}) as the action
of a pair of linear functionals $(\wh{\vform},\wh{\sform})$ acting
on $\Psi$ such that 
\begin{equation}
\wh{\vform}(\ptnl)=\int_{\gO}\ptnl\vform\dee V,\qquad\wh{\sform}(\ptnl\resto{\bdry\gO})=\int_{\bdry\gO}\ptnl\sform\dee A.
\end{equation}
We may now write the left-hand side of (\ref{eq:weak-balance}) in
the form
\begin{equation}
\int_{\gO}\ptnl\vform\dee V+\int_{\bdry\gO}\ptnl\,\sform\dee A=(\wh{\vform},\wh{\sform})(\ptnl,\ptnl\resto{\bdry\gO}).
\end{equation}

Let $\trace(\psi)=\psi\resto{\bdry\gO}$, and let $\gd:=(\idnt,\trace)$,
that is,
\begin{equation}
\gd(\ptnl)=(\idnt,\trace)(\ptnl)=(\ptnl,\ptnl\resto{\bdry\gO}),
\end{equation}
and we may further write
\begin{equation}
\int_{\gO}\ptnl\vform\dee V+\int_{\bdry\gO}\ptnl\,\sform\dee A=(\wh{\vform},\wh{\sform})(\gd(\ptnl)).
\end{equation}
Next, consider the dual mapping $\gd^{*}$ of $\gd$ such that
\begin{equation}
(\wh{\vform},\wh{\sform})(\gd(\ptnl))=\gd^{*}(\wh{\vform},\wh{\sform})(\ptnl),
\end{equation}
then, we finally write the left-hand side of (\ref{eq:weak-balance})
as $\gd^{*}(\wh{\vform},\wh{\sform})(\ptnl)$.

Here, we view $(\wh{\vform},\wh{\sform})$ as general continuous linear
functionals in the appropriate dual spaces. For the operator $\gd$
to be well-defined, a trace mapping, $\trace$, of non-smooth elements
in $\Psi$, extending the restriction of continuous functions on $\cl{\gO}$
to $\bdry\gO$, should be available.

The right hand side of (\ref{eq:weak-balance}) may also be expressed
in terms of two linear functionals
\begin{equation}
\wh{\pform}(\ptnl)=\int_{\gO}\ptnl\pform\dee V,\qquad\wh{\vf}(\nabla\ptnl)=\int_{\bdry\gO}\nabla\ptnl\cdot\vf\dee V.
\end{equation}
Thus, 
\begin{equation}
\int_{\gO}\ptnl\pform\dee V+\int_{\gO}\nabla\ptnl\cdot\flow\dee V=(\wh{\pform},\wh{\vf})(\ptnl,\nabla\ptnl).
\end{equation}

Consider the linear operator
\begin{equation}
j=(\idnt,\nabla):\psi\lmt(\ptnl,\nabla\ptnl),
\end{equation}
so that the right hand side of (\ref{eq:weak-balance}) may be written
as
\begin{equation}
(\wh{\pform},\wh{\vf})(j(\ptnl))=j^{*}(\wh{\pform},\wh{\vf})(\ptnl).
\end{equation}

It is concluded that (\ref{eq:weak-balance}) may be rewritten in
the form
\begin{equation}
\gd^{*}(\wh{\vform},\wh{\sform})(\ptnl)=j^{*}(\wh{\pform},\wh{\vf})(\ptnl),
\end{equation}
 and as $\ptnl$ is arbitrary, 
\begin{equation}
\gd^{*}(\wh{\vform},\wh{\sform})=j^{*}(\wh{\pform},\wh{\vf})=\fc\label{eq:weak_Balance}
\end{equation}
where $\fc$ is a linear functional on $\Psi$.

This construction is illustrated in the following diagrams.
\begin{equation}
\xymatrix{(\ptnl,\trace(\ptnl)) & \ptnl\ar@{|->}[r]\sp(0.35){j}\ar@{|->}[l]\sb(0.35){\gd} & (\ptnl,\nabla\ptnl)\\
(\wh{\vform},\wh{\sform})\ar@{|->}[r]\sp(0.55){\gd^{*}} & \fc & \ar@{|->}[l]\sb(0.55){j^{*}}(\wh{\pform},\wh{\flow})
}
\label{eq:diagram-elements-2}
\end{equation}
To apply this framework, one has to define appropriately the function
space $\Psi$ so that the mappings $\trace$ and $j$ are well-defined.

With this framework, we can describe the optimization problem as finding
an optimal pair, $(\wh{\pform},\wh{\vf})$, the satisfies Equation
(\ref{eq:weak_Balance}) for a given pair $(\wh{\vform},\wh{\sform})$.

In formulating the optimization problem, we will consider only the
case where $\wh{\pform}=0$. This is motivated by the assumption that
while, in principle, one may be able to control the velocity distribution
$\vf$, one will not be able to introduce a source distribution due
to practical constraints. This implies that the diagrams above are
modified into
\begin{equation}
\xymatrix{(\ptnl,\trace(\ptnl)) & \ptnl\ar@{|->}[r]\sp(0.35){\nabla}\ar@{|->}[l]\sb(0.35){\gd} & \nabla\ptnl\\
(\wh{\vform},\wh{\sform})\ar@{|->}[r]\sp(0.55){\gd^{*}} & \fc & \ar@{|->}[l]\sb(0.55){\nabla^{*}}\wh{\flow}.
}
\label{eq:diagram-elements-2-1}
\end{equation}
The weak balance condition is accordingly modified to
\begin{equation}
\gd^{*}(\wh{\vform},\wh{\sform})=\nabla^{*}(\wh{\vf})=\fc.\label{eq:weak_Balance-1}
\end{equation}

\section{Analytic Preliminaries\label{sec:Analytic-Preliminaries}}

In this section, we summarize the basic properties of Sobolev spaces,
which make them suitable candidates for the function spaces needed
to materialize the framework described roughly in Section \ref{sec:Weak-Formulation}.

Let $\gO\subset\rthree$ be a bounded open set with Lipschitz boundary,
and let $W^{m,p}(\gO)$ be the Sobolev space of functions defined
on $\gO$, of which the $m$-th distributional derivatives are $L^{p}$-integrable.
For a multi-index $\ga=(\ga_{1},\ga_{2},\ga_{3})$, set $\abs{\ga}=\ga_{1}+\ga_{2}+\ga_{3}$.
Given $\psi\in W^{m,p}(\gO)$, let
\begin{equation}
\bdry_{\ga}\psi:=\frac{\partial^{\abs{\ga}}\psi}{\bdry x_{1}^{\ga_{1}}\bdry x_{2}^{\ga_{2}}\bdry x_{3}^{\ga_{3}}}
\end{equation}
and 
\begin{equation}
\nabla_{m}\psi=\{\bdry_{\ga}\psi\mid\abs{\ga}=m\}.
\end{equation}

The Sobolev spaces are Banach spaces using the norms
\begin{equation}
\norm{\psi}_{W^{m,p}(\gO)}=\sum_{0\les\abs{\ga}\les m}\norm{\bdry_{\ga}\psi}_{L^{p}(\gO)},
\end{equation}
where alternative equivalent norms may be given.

If $mp>3$, or $m=3$ and $p=1$, then, there is an embedding (see
\cite[p. 85]{Adams})
\begin{equation}
W^{m,p}(\gO)\tto C^{0}(\ol{\gO}).\label{eq:emb_cont}
\end{equation}
If $mp<3$, and $q$ satisfies $p\les q\les2p/(3-mp)$, then there
is a linear and continuous trace operator (see \cite[p. 164]{Adams}),
\begin{equation}
\trace:W^{m,p}(\gO)\tto L^{q}(\bdry\gO).
\end{equation}
It is noted that for the case $mp>3$, the existence of the trace
operator is implied by (\ref{eq:emb_cont}), as functions can be extended
continuously to the boundary.

In the following, we describe the construction of the subspaces $\dot{W}^{m,p}(\gO)$
(denoted in \cite{Mazja} by $\dot{L}_{p}^{m}(\gO)$) of the Sobolev
spaces.

\begin{myprop}[\textbf{Poincar\'e, Sobolev}]

Let $P^{m}$ be the subspace of polynomials in $\gO$ of degree smaller
or equal to $m$. Then, there is a unique continuous and linear projection
\begin{equation}
\pi_{P}:W^{m,p}(\gO)\tto P^{m-1}
\end{equation}
such that 
\begin{equation}
\sum_{k=0}^{m-1}\norm{\nabla_{k}(u-\pi_{P}(u))}_{L^{p}(\gO)}\les C\norm{\nabla_{m}u}_{L^{p}(\gO)},
\end{equation}
for a constant $C>0$ and all $u\in W^{m,p}$. Here, 
\begin{equation}
\norm{\nabla_{m}u}_{L^{p}(\gO)}:=\sum_{\abs{\ga}=m}\norm{\partial_{\ga}u}_{L^{p}(\gO)}.
\end{equation}

In particular, for the case $m=1$, $\pi_{P}:W^{1,p}(\gO)\tto P^{0}=\reals$
is given by 
\begin{equation}
u\lmt\frac{1}{\abs{\gO}}\int_{\gO}u\dee V,
\end{equation}
where $\abs{\gO}$ denotes the volume of $\gO$.

\end{myprop}
\begin{proof}
See \cite[p. 22]{Mazja} and \cite[Section 13.2]{Leoni2010}. 
\end{proof}

\begin{myprop}\label{prop:norm-dot}Let 
\begin{equation}
\dot{W}^{m,p}(\gO):=W^{m,p}(\gO)/P^{m-1}.
\end{equation}
For any $\dot{u}\in\dot{W}^{m,p}(\gO)$, set
\begin{equation}
\norm{\dot{u}}_{\dot{W}^{m,p}(\gO)}:=\norm{\nabla_{m}u}_{L^{p}(\gO)},
\end{equation}
where $u\in W^{m,p}(\gO)$ is any representative of the equivalence
class $\dot{u}$. Then, $\norm{\cdot}_{\dot{W}^{m,p}}$ is a norm
on $\dot{W}^{m,p}(\gO)$.

\end{myprop}
\begin{proof}
See \cite[p. 24]{Mazja}.
\end{proof}
\begin{cor}
\label{cor:Wdot-subspace}We may identify $\dot{W}^{m,p}(\gO)$ with
the subspace $\kernel\pi_{P}$ of $W^{m,p}(\gO)$. Thus, the projection
on the quotient space
\begin{equation}
\dot{\pi}:W^{m,p}(\gO)\tto\dot{W}^{m,p}(\gO):=W^{m,p}(\gO)/P^{m-1}
\end{equation}
is given by the mapping
\begin{equation}
u\lmt u-\pi_{P}(u).
\end{equation}
The right inverse of $\dot{\pi}$ is therefore the inclusion 
\begin{equation}
\xymatrix{\kernel\pi_{P}\ar@{^{(}->}[r]^{\incl} & W^{m,p}(\gO).}
\end{equation}

In particular, for the case $m=1$, $\dot{W}^{1,p}(\gO)$ may be identified
with the subspace of $W^{1,p}(\gO)$ containing functions of zero
average.
\end{cor}

\begin{cor}
\label{cor:Nabla-dot}Consider the linear and continuous
\begin{equation}
\nabla:W^{m,p}(\gO)\tto W^{m-1,p}(\gO)^{3},\qquad u\lmt\nabla u.
\end{equation}
Then, the restriction of $\nabla$ to $\dot{W}^{m,p}(\gO)$, viewed
as $\kernel\pi_{P}$, induces the mapping 
\begin{equation}
\dot{\nabla}:\dot{W}^{m,p}(\gO)\tto\dot{W}^{m-1,p}(\gO)^{3}
\end{equation}
(not to be confused with time-derivative) defined as to make the following
diagram commutative,
\begin{equation}
\xymatrix{W^{m,p}(\gO)\ar[r]\sp(0.46){\nabla}\ar@<-0.4ex>[d]^{\dot{\pi}} & W^{m-1,p}(\gO)^{3}\ar@<-0.4ex>[d]^{\dot{\pi}}\\
\dot{W}^{m,p}(\gO)\ar[r]\sp(0.46){\dot{\nabla}}\ar@{^{(}->}@<2ex>[u]^{\incl} & \dot{W}^{m-1,p}(\gO)^{3}\ar@{^{(}->}@<2ex>[u]^{\incl}.
}
\end{equation}
Moreover, $\dot{\nabla}$ is injective, linear, and norm-preserving.
\end{cor}

\section{$L^{q}$-Optimization\label{sec:Lq-opt}}

In this section, for a given bounded open subset $\gO\subset\rthree$
having a Lipschitz boundary, and for a given flux field on the boundary
and a given density rate field, we consider flux vector fields that
are optimal in the $L^{q}$-norm. Here, we include the case $q=\infty$
for which the essential supremum of the flux field is minimized.

The space of potential functions, $\Psi$, is identified in this case
with the space $\dot{W}^{1,p}(\gO)$, where $q=p/(p-1)$ is the dual
exponent to $p$. It is recalled that $\dot{W}^{1,p}(\gO)$ may be
viewed as the subspace of $W^{1,p}(\gO)$ containing functions of
zero average.

Note that as a special case of Corollary \ref{cor:Nabla-dot}, the
mapping 
\begin{equation}
\dot{\nabla}:\dot{W}^{1,p}(\gO)\tto L^{p}(\gO)^{3}
\end{equation}
is well-defined, linear, injective, and norm-preserving. It is also
noted that $\image\dot{\nabla}$ is a proper subspace of $L^{p}(\gO)$
which is not dense.

\begin{myprop}\label{prop:existence-L^p}Let $\fc\in\dot{W}^{1,p}(\gO)^{*}$
be a bounded linear functional. Then, $\fc$ may be represented in
the form
\begin{equation}
\fc=\dot{\nabla}^{*}(\vf),\qquad\vf\in L^{q}(\gO)^{3},
\end{equation}
where 
\begin{equation}
\dot{\nabla}^{*}:L^{q}(\gO)^{3}\isom L^{p}(\gO)^{*3}\tto\dot{W}^{1,p}(\gO)^{*}
\end{equation}
is the mapping dual to $\dot{\nabla}$.

\end{myprop}
\begin{proof}
Since $\dot{\nabla}$ is injective and norm-preserving the mapping
\begin{equation}
\dot{\nabla}^{-1}:\image\dot{\nabla}\subset L^{p}(\gO)^{3}\tto\dot{W}^{1,p}(\gO)
\end{equation}
is well-defined, and it has same properties. Thus, given $\fc\in\dot{W}^{1,p}(\gO)^{*}$,
\begin{equation}
\vf_{0}:=\fc\comp\dot{\nabla}^{-1}:\image\dot{\nabla}\subset L^{p}(\gO)^{3}\tto\reals
\end{equation}
is a bounded linear functional on $\image\dot{\nabla}$ as illustrated
in the following diagram. 
\begin{equation}
\xymatrix{\reals & \dot{W}^{1,p}(\gO)\ar[r]\sp(0.45){\dot{\nabla}}\ar[l]\sb(0.57){\fc} & \image\dot{\nabla}\hphantom{.}\ar@{^{(}->}[r]\sp(0.5){\incl}\ar@/_{2pc}/[ll]_{\vf_{0}=\fc\comp\dot{\nabla}^{-1}}\ar@/^{1.5pc}/[l]^{\dot{\nabla}^{-1}} & L^{p}(\gO)^{3}.}
\end{equation}

By the Hahn-Banach theorem, $\vf_{0}$ may be extended to some (non-unique)
bounded linear functional $\vf\in L^{q}(\gO)^{3}\isom L^{p}(\gO)^{*3}$.
For each $\ptnl\in\dot{W}^{1,p}(\gO)$,
\begin{equation}
\begin{split}\vf(\dot{\nabla}(\ptnl)) & =\vf_{0}(\dot{\nabla}(\ptnl)),\\
 & =\fc\comp\dot{\nabla}^{-1}(\dot{\nabla}(\ptnl)),\\
 & =\fc(\ptnl),
\end{split}
\end{equation}
and we conclude that $\fc=\vf\comp\dot{\nabla}=\dot{\nabla}^{*}(\vf)$,
as asserted.
\end{proof}

\begin{myprop}Let $\fc\in\dot{W}^{1,p}(\gO)^{*}$ be given. Then,
there is an optimal element $\vf^{\opt}\in L^{q}(\gO)^{3}$ that represents
$\fc$ in the form $\fc=\dot{\nabla}^{*}(\vf^{\opt})$. That is, 
\begin{equation}
\go_{\fc}:=\inf\{\norm{\vf}_{L^{q}(\gO)^{3}}\mid\vf\in L^{q}(\gO)^{3},\,\fc=\dot{\nabla}^{*}(\vf)\}=\norm{\vf^{\opt}}_{L^{q}(\gO)^{3}}.
\end{equation}
Moreover, 
\begin{equation}
\go_{\fc}=\norm{\vf^{\opt}}_{L^{q}(\gO)^{3}}=\norm{\fc}_{(\dot{W}^{1,p}(\gO))^{*}}.
\end{equation}

\end{myprop}
\begin{proof}
Using the same notation as in the proof above, it is noted that in
general, for an extension $\vf\in L^{q}(\gO)^{3}$ of $\vf_{0}=\fc\comp\dot{\nabla}^{-1}$,
one has
\begin{equation}
\begin{split}\norm{\vf} & =\sup\left\{ \frac{\vf(\gf)}{\norm{\gf}_{L^{p}(\gO)^{3}}}\mid\gf\in L^{p}(\gO)^{3},\,\gf\ne0\right\} ,\\
 & \ges\sup\left\{ \frac{\vf_{0}(\dot{\nabla}(\ptnl))}{\norm{\dot{\nabla}(\ptnl)}_{L^{p}(\gO)^{3}}}\mid\ptnl\in\dot{W}^{1,p}(\gO),\,\ptnl\ne0\right\} ,\\
 & =\sup\left\{ \frac{\fc(\ptnl)}{\norm{\ptnl}_{\dot{W}^{1,p}(\gO)}}\mid\ptnl\in\dot{W}^{1,p}(\gO),\,\ptnl\ne0\right\} .
\end{split}
\end{equation}
Hence, for an element $\vf\in L^{q}(\gO)^{3}$ representing $\fc$,
\begin{equation}
\norm{\vf}_{L^{p}(\gO)^{3}}\ges\norm{\fc}_{(\dot{W}^{1,p}(\gO))^{*}}.
\end{equation}

However, by the Hahn-Banach theorem, there is a norm preserving extension,
$\vf\hb$, so that
\begin{equation}
\begin{split}\norm{\vf\hb} & _{L^{q}(\gO)^{3}}=\sup\left\{ \frac{\vf\hb(\gf)}{\norm{\gf}_{L^{p}(\gO)^{3}}}\mid\gf\in L^{p}(\gO)^{3},\,\gf\ne0\right\} ,\\
 & =\sup\left\{ \frac{\vf_{0}(\dot{\nabla}(\ptnl))}{\norm{\dot{\nabla}(\ptnl)}_{L^{p}(\gO)^{3}}}\mid\ptnl\in\dot{W}^{1,p}(\gO),\,\ptnl\ne0\right\} ,\\
 & =\sup\left\{ \frac{\fc(\ptnl)}{\norm{\ptnl}_{\dot{W}^{1,p}(\gO)}}\mid\ptnl\in\dot{W}^{1,p}(\gO),\,\ptnl\ne0\right\} 
\end{split}
\end{equation}
so that $\vf^{\opt}=\vf\hb$, is an optimal element representing $\fc$
and $\norm{\vf^{\opt}}_{L^{q}(\gO)^{3}}=\norm{\fc}_{\dot{W}^{1,p}(\gO)^{*}}$.
\end{proof}
Next, it is recalled that the trace mapping
\begin{equation}
\trace:W^{1,p}(\gO)\tto L^{p}(\bdry\gO)
\end{equation}
is a well defined bounded linear operator. Consider the restriction
of the trace mapping to $\dot{W}^{1,p}(\gO)$, thus obtaining the
trace mapping on the subspace of potential functions having zero averages,
\begin{equation}
\dot{\trace}:\dot{W}^{1,p}(\gO)\tto L^{p}(\bdry\gO).
\end{equation}
We may define therefore the bounded\, linear operator
\begin{equation}
\gd:=(\incl,\dot{\trace}):\dot{W}^{1,p}(\gO)\tto L^{p}(\gO)\times L^{p}(\bdry\gO),
\end{equation}
where 
\begin{equation}
\incl:\dot{W}^{1,p}(\gO)\tto L^{p}(\gO)
\end{equation}
 is the natural embedding.

It is concluded that every element $(\vform,\sform)\in(L^{p}(\gO)\times L^{p}(\bdry\gO))^{*}\isom L^{q}(\gO)\times L^{q}(\bdry\gO)$
induces an element $\fc=\gd^{*}(\vform,\sform)\in\dot{W}^{1,p}(\gO)^{*}$.
Summarizing, we have the following diagrams
\begin{equation}
\xymatrix{L^{p}(\gO)\times L^{p}(\bdry\gO) & \dot{W}^{1,p}(\gO)\ar@{->}[r]\sp(0.45){\dot{\nabla}}\ar@{->}[l]\sb(0.42){\gd} & L^{p}(\gO)^{3}\\
L^{q}(\gO)^{*}\times L^{q}(\bdry\gO)\ar@{->}[r]\sp(0.6){\gd^{*}} & \dot{W}^{1,p}(\gO)^{*} & \ar@{->}[l]\sb(0.5){\dot{\nabla}^{*}}L^{q}(\gO)^{*}
}
\label{eq:diagram-1-1}
\end{equation}

\begin{myprop}Let $(\vform,\sform)\in(L^{p}(\gO)\times L^{p}(\bdry\gO))^{*}\isom L^{q}(\gO)\times L^{q}(\bdry\gO)$.
Then, there is an optimal element $\vf^{\opt}\in L^{q}(\gO)^{3}$
that satisfies the weak balance equation
\begin{equation}
\gd^{*}(\vform,\sform)=\dot{\nabla}^{*}(\vf).
\end{equation}
That is, 
\begin{equation}
\go_{(\vform,\sform)}:=\inf\{\norm{\vf}_{L^{q}(\gO)^{3}}\mid\vf\in L^{q}(\gO)^{3},\,\gd^{*}(\vform,\sform)=\dot{\nabla}^{*}(\vf)\}=\norm{\vf^{\opt}}_{L^{q}(\gO)^{3}}.
\end{equation}
Moreover, 
\begin{multline}
\go_{(\vform,\sform)}=\norm{\vf^{\opt}}_{L^{q}(\gO)^{3}}=\norm{\gd^{*}(\vform,\sform)}_{\dot{W}^{1,p}(\gO)^{*}}\\
=\sup\left\{ \frac{\int_{\gO}\vform\ptnl\dee{V+\int_{\bdry\gO}\sform\ptnl\dee A}}{\left[\int_{\gO}\abs{\nabla\ptnl}^{p}\dee V\right]^{1/p}}\mid\ptnl\in\dot{W}^{1,p}(\gO)\right\} .
\end{multline}

\end{myprop}
\begin{proof}
One simply combines the foregoing observations and uses the definitions
of the norms in the various spaces.
\end{proof}
\begin{rem}
Note that since smooth functions are dense in the Sobolev spaces,
it is sufficient to take the supremum above over smooth functions
having zero averages.
\end{rem}

\section{Dissipation Optimization in $W^{1,2}(\protect\gO)$\label{sec:Dissipation-Optimization-in-W12}}

An important criterion for the optimization of a flux field is maximizing
the efficiency of the flow by minimizing the dissipation of energy.
The rate of dissipation of energy is taken for simplicity as (see
\cite[p. 580]{Lamb32})
\[
\int_{\gO}\abs{\nabla\vf}^{2}\dee V=\int_{\gO}\nabla\vf_{i}\cdot\nabla\vf_{i}\dee{V.}
\]

In terms of the notation of Section \ref{sec:Analytic-Preliminaries},
we wish to minimize 
\begin{equation}
\norm{\nabla\vf}_{L^{2}(\gO)^{3}}=\norm{\dot{\vf}}_{\dot{W}^{1,2}(\gO)^{3}}.
\end{equation}
The following construction is analogous to that of the foregoing section
with some necessary adaptations.

\subsection{The space of potential functions}

The space of potential functions is taken as $\dot{W}^{2,2}(\gO)$,
and it is viewed as a subspace of $W^{2,2}(\gO)$. It follows from
Proposition \ref{prop:norm-dot} that 
\begin{equation}
\norm{\ptnl}_{\dot{W}^{2,2}(\gO)}=\norm{\nabla_{2}u}_{L^{2}(\gO)}=\left[\int_{\gO}u_{,ij}u_{,ij}\dee V\right]^{1/2},
\end{equation}
where $u$ is any representative of the equivalence class $\ptnl$,
is a norm on $\dot{W}^{2,2}(\gO)$. In fact, in what follows, we will
write $\ptnl_{,ij}$ for $u_{,ij}$ as the expression is independent
of the particular element representing $\ptnl$.

Evidently, $\dot{W}^{2,2}(\gO)$ is a Hilbert space (see the proof
of the completeness in \cite{Mazja}) with the inner product
\begin{equation}
\innp{\ptnl}{\ptnl'}=\int_{\gO}\ptnl_{,ij}\ptnl'_{,ij}\dee V.\label{eq:InnerP-W22}
\end{equation}

It is recalled that by the Sobolev embedding theorem, as $mp=2\cdot2>3=\dim\rthree$,
(see \cite[p. 85]{Adams}), there is an embedding 
\begin{equation}
\eps:W^{2,2}(\gO)\tto C^{0}(\ol{\gO}).\label{eq:embedding-Sobolev-2}
\end{equation}
Thus, every function in $W^{2,2}(\gO$) is continuous and may be extended
continuously to the Lipschitz boundary, $\bdry\gO$. This applies,
in particular, to the subspace $\dot{W}^{2,2}(\gO)$. Thus, we have
a bounded linear operator 
\begin{equation}
\gd=(\idnt,\dot{\trace}):\dot{W}^{2,2}(\gO)\tto C^{0}(\gO)\times C^{0}(\bdry\gO),\qquad\ptnl\lmt(\ptnl,\ptnl\resto{\bdry\gO}).
\end{equation}

A pair $(\wh{\vform},\wh{\sform})\in C^{0}(\gO)^{*}\times C^{0}(\bdry\gO)^{*}$,
will therefore induce an element $\fc\in\dot{W}^{2,2}(\gO)^{*}$ by
\begin{equation}
\fc=\gd^{*}(\wh{\vform},\wh{\sform}).
\end{equation}
It is noted that as linear functionals of spaces of $C^{0}$-functions,
$\wh{\vform}$ and $\wh{\sform}$ are represented by measures $\vform$
and $\sform$ on $\gO$ and $\bdrr$, respectively. Thus,
\begin{equation}
\fc(\ptnl)=\int_{\gO}\ptnl\dee{\vform}+\int_{\bdry\gO}\ptnl\dee{\sform}.\label{eq:func-vol-sur}
\end{equation}

\subsection{Representation and optimization}

Next, Corollary \ref{cor:Nabla-dot} implies that we have a bounded,
linear operator
\begin{equation}
\dot{\nabla}:\dot{W}^{2,2}(\gO)\tto\dot{W}^{1,2}(\gO)^{3},
\end{equation}
which is injective and norm-preserving. Hence, in analogy with the
the construction for $W^{1,q}(\gO)$-optimization we have the following.

\begin{myprop}\label{prop:Representating_F}Let $\fc\in\dot{W}^{2,2}(\gO)^{*}$
be a bounded linear functional. Then, $\fc$ may be represented in
the form
\begin{equation}
\fc=\dot{\nabla}^{*}(\vf),\qquad\vf\in\dot{W}^{1,2}(\gO)^{3*},
\end{equation}
where 
\begin{equation}
\dot{\nabla}^{*}:\dot{W}^{1,2}(\gO)^{3*}\tto\dot{W}^{2,2}(\gO)^{*}
\end{equation}
is the mapping dual to $\dot{\nabla}$.

\end{myprop}
\begin{proof}
In analogy with Proposition \ref{prop:existence-L^p} we have the
following diagram
\begin{equation}
\xymatrix{\reals & \dot{W}^{2,2}(\gO)\ar[r]\sp(0.45){\dot{\nabla}}\ar[l]\sb(0.57){\fc} & \image\dot{\nabla}\hphantom{.}\ar@{^{(}->}[r]\sp(0.5){\incl}\ar@/_{2pc}/[ll]_{\vf_{0}=\fc\comp\dot{\nabla}^{-1}}\ar@/^{1.5pc}/[l]^{\dot{\nabla}^{-1}} & \dot{W}^{1,2}(\gO)^{3}.}
\label{eq:Representing_F}
\end{equation}

By the Hahn-Banach theorem, $\vf_{0}$ may be extended a bounded linear
functional $\vf\in\dot{W}^{1,p}(\gO)^{*3}$. For each $\ptnl\in\dot{W}^{2,2}(\gO)$,
\begin{equation}
\begin{split}\vf(\dot{\nabla}(\ptnl)) & =\vf_{0}(\dot{\nabla}(\ptnl)),\\
 & =\fc\comp\dot{\nabla}^{-1}(\dot{\nabla}(\ptnl)),\\
 & =\fc(\ptnl),
\end{split}
\end{equation}
and we conclude that $\fc=\vf\comp\dot{\nabla}=\dot{\nabla}^{*}(\vf)$,
as asserted.
\end{proof}

\begin{myprop}\label{prop:optimality}Let $\fc\in\dot{W}^{2,2}(\gO)^{*}$
be given. Then, there is an optimal element $\vf^{\opt}\in\dot{W}^{1,2}(\gO)^{3*}$
that represents $\fc$ in the form $\fc=\dot{\nabla}^{*}(\vf^{\opt})$.
That is, 
\begin{equation}
\go_{\fc}:=\inf\{\norm{\vf}_{\dot{W}^{1,2}(\gO)^{3*}}\mid\vf\in\dot{W}^{1,2}(\gO)^{3*},\,\fc=\dot{\nabla}^{*}(\vf)\}=\norm{\vf^{\opt}}_{\dot{W}^{1,2}(\gO)^{3*}}.
\end{equation}
Moreover, 
\begin{equation}
\go_{\fc}=\norm{\vf^{\opt}}_{\dot{W}^{1,2}(\gO)^{3*}}=\norm{\fc}_{(\dot{W}^{2,2}(\gO))^{*}}.
\end{equation}

Specifically, for the case $\fc=\gd^{*}(\wh{\vform},\wh{\sform})$
considered here,
\begin{equation}
\go_{\fc}=\sup\left\{ \frac{\int_{\gO}\ptnl\dee{\vform}+\int_{\bdrr}\ptnl\dee{\sform}}{\left[\int_{\gO}\ptnl_{,ij}\ptnl_{,ij}\dee V\right]^{1/2}}\mid\ptnl\in\dot{W}^{2.2}(\gO)\right\} .
\end{equation}

\end{myprop}
\begin{proof}
Again, in analogy with the previous section, for an element $\vf\in\dot{W}^{1,2}(\gO)^{3*}$
representing $\fc$, 
\begin{equation}
\norm{\vf}_{\dot{W}^{1,2}(\gO)^{3*}}\ges\norm{\fc}_{(\dot{W}^{2,2}(\gO))^{*}}.
\end{equation}

However, by the Hahn-Banach theorem, there is a norm preserving extension,
$\vf\hb$, so that
\begin{equation}
\begin{split}\norm{\vf\hb} & _{\dot{W}^{1,2}(\gO)^{3*}}=\norm{\fc}_{\dot{W}^{2,2}(\gO)^{*}}.\end{split}
\end{equation}
Hence, $\vf^{\opt}=\vf\hb$, is an optimal element representing $\fc$
and $\norm{\vf^{\opt}}_{\dot{W}^{1,2}(\gO)^{3*}}=\norm{\fc}_{(\dot{W}^{2,2}(\gO))^{*}}$.
\end{proof}

\subsection{Implication of the Hilbert space structure}

As noted above, under the settings of this section, where potential
functions are elements of $\dot{W}^{2,2}(\gO)$, fluxes are elements
$\vf\in\dot{W}^{1,2}(\gO)^{3*}$. 

Since $\dot{W}^{2,2}(\gO)$ is a Hilbert space under the inner product
(\ref{eq:InnerP-W22}), by the Riesz representation theorem for Hilbert
spaces, for each element $\fc\in\dot{W}^{2,2}(\gO)^{*}$, there is
an element $\ol F\in\dot{W}^{2,2}(\gO)$ such that 
\begin{equation}
\fc(\ptnl)=\inner{\ol F}{\ptnl}=\int_{\gO}\ol{\fc}_{,ij}\ptnl_{,ij}\dee V,
\end{equation}
and
\begin{equation}
\norm{\fc}_{\dot{W}^{2,2}(\gO)^{*}}^{2}=\int_{\gO}\ol{\fc}_{,ij}\ol{\fc}_{,ij}\dee V.
\end{equation}

In the sequel, we will adopt the scheme of notation by which for an
element $h^{*}$ in the dual space, $H^{*}$, of a Hilbert space,
$H$, $\ol{h^{*}}$ will denote the element of $H$ representing $h^{*}$. 

The space $\dot{W}^{1,2}(\gO)^{3}$ is a Hilbert space with the inner
product
\begin{equation}
\innp{\gf}{\gf'}=\int_{\gO}\gf_{i,j}\gf'_{i,j}\dee V=\int_{\gO}\nabla\gf_{i}\cdot\nabla\gf'_{i}\dee V.
\end{equation}
Thus, any element $\vf\in\dot{W}^{1,2}(\gO)^{3*}$ may be represented
by an element of $\ol{\vf}\in\dot{W}^{1,2}(\gO)^{3}$, in the form
\[
\vf(\gf)=\inner{\ol{\vf}_{i}}{\gf_{i}}=\int_{\gO}\ol{\vf}_{i,j}\gf_{i,j}\dee V.
\]

Moreover, 
\begin{equation}
\norm{\vf}_{\dot{W}^{1,2}(\gO)^{3*}}^{2}=\norm{\ol{\vf}}_{\dot{W}^{1,2}(\gO)^{3}}^{2}=\inner{\ol{\vf}}{\ol{\vf}}=\int_{\gO}\ol{\vf}_{i,j}\ol{\vf}_{i,j}\dee V.
\end{equation}
Thus, minimizing the dual norm of $\vf$ is equivalent to minimizing
the dissipation, as required. 

Using these structures, the mapping $\dot{\nabla}^{*}:\dot{W}^{1,2}(\gO)^{3*}\tto\dot{W}^{2,2}(\gO)^{*}$
is represented by the mapping
\begin{equation}
\dot{\nabla}\trps:\dot{W}^{1,2}(\gO)^{3}\tto\dot{W}^{2,2}(\gO)
\end{equation}
satisfying
\begin{equation}
\ol{\dot{\nabla}^{*}(\vf)}=\dot{\nabla}\trps(\ol{\vf}).
\end{equation}
(We recall that the bar over $\dot{\nabla}^{*}(\vf)$ indicates the
element in the Hilbert space that represent the bounded linear functional.)
It follows that
\begin{equation}
\bigl\langle\dot{\nabla}\trps(\ol{\vf})\bigr\rangle=\inner{\ol{\vf}}{\dot{\nabla}(\psi)}.
\end{equation}

We can therefore replace the diagram
\begin{equation}
\xymatrix{C^{0}(\gO)\times C^{0}(\bdry\gO) & \dot{W}^{2,2}(\gO)\ar@{->}[r]\sp(0.45){\dot{\nabla}}\ar@{->}[l]\sb(0.42){\gd} & \dot{W}^{1,2}(\gO)^{3}\\
C^{0}(\gO)^{*}\times C^{0}(\bdry\gO)^{*}\ar@{->}[r]\sp(0.6){\gd^{*}} & \dot{W}^{2,2}(\gO)^{*} & \ar@{->}[l]\sb(0.5){\dot{\nabla}^{*}}\dot{W}^{1,2}(\gO)^{3*}
}
\end{equation}
by the diagram
\begin{equation}
\xymatrix{C^{0}(\gO)\times C^{0}(\bdry\gO) & \dot{W}^{2,2}(\gO)\ar@{->}[r]\sp(0.45){\dot{\nabla}}\ar@{->}[l]\sb(0.42){\gd} & \dot{W}^{1,2}(\gO)^{3}\\
C^{0}(\gO)^{*}\times C^{0}(\bdry\gO)^{*}\ar@{->}[r]\sp(0.6){\gd\trps} & \dot{W}^{2,2}(\gO) & \ar@{->}[l]\sb(0.5){\dot{\nabla}\trps}\dot{W}^{1,2}(\gO)^{3},
}
\end{equation}
where the mapping $\gd\trps:C^{0}(\gO)^{*}\times C^{0}(\bdry\gO)^{*}\tto\dot{W}^{2,2}(\gO)$,
satisfies 
\begin{equation}
\gd\trps(\wh{\vform},\wh{\sform})=\ol{\gd^{*}(\wh{\vform},\wh{\sform})}.
\end{equation}
Let $\fc=\gd^{*}(\wh{\vform},\wh{\sform})$, then,
\begin{equation}
\inner{\ol{\fc}}{\ptnl}=\int_{\gO}\ol{\fc}_{,ij}\ptnl_{ij}\dee V=(\wh{\vform},\wh{\sform})(\gd(\ptnl))=\int_{\gO}\ptnl\dee{\vform}+\int_{\bdry\gO}\ptnl\dee{\sform.}
\end{equation}

The weak balance equation may be reformulated now as
\begin{equation}
\gd\trps(\wh{\vform},\wh{\sform})=\dot{\nabla}\trps(\ol{\vf})
\end{equation}

\begin{myprop}Let $(\vform,\sform)$ be a pair of (Radon) measures
on $\gO$ and $\bdry\gO$, respectively. 
\begin{enumerate}
\item There is a unique dissipation minimizing element $\ol{\vf}^{\opt}\in\dot{W}^{1,2}(\gO)^{3}$
that satisfies the weak balance equation
\begin{equation}
\gd^{*}(\wh{\vform},\wh{\sform})=\dot{\nabla}^{*}(\vf)\qquad\text{or alternatively,}\qquad\gd\trps(\wh{\vform},\wh{\sform})=\dot{\nabla}\trps(\ol{\vf}).
\end{equation}
That is, 
\begin{equation}
\go_{(\vform,\sform)}:=\inf\left\{ \norm{\ol{\vf}}_{\dot{W}^{1,2}(\gO)^{3}}^{2}=\int_{\gO}\ol{\vf}_{i,j}\ol{\vf}_{i,j}\dee V\right\} =\norm{\ol{\vf}^{\opt}}_{\dot{W}^{1,2}(\gO)^{3}}
\end{equation}
where the infimum is taken over all $\vf\in\dot{W}^{1,2}(\gO)^{3}$,
$\gd^{*}(\vform,\sform)=\dot{\nabla}^{*}(\vf)$. Moreover, 
\begin{equation}
\go_{(\vform,\sform)}=\norm{\vf^{\opt}}_{\dot{W}^{1,2}(\gO)^{3}}=\norm{\gd^{*}(\wh{\vform},\wh{\sform})}_{\dot{W}^{2,2}(\gO)^{*}},
\end{equation}
is given by
\begin{equation}
\go_{(\vform,\sform)}=\sup\left\{ \frac{\int_{\gO}\ptnl\dee{\vform+\int_{\bdry\gO}\ptnl\dee{\sform}}}{\left[\int_{\gO}\ptnl_{,ij}\ptnl_{,ij}\dee V\right]^{1/2}}\mid\ptnl\in\dot{W}^{2,2}(\gO)\right\} .
\end{equation}
\end{enumerate}
\end{myprop}
\begin{proof}
Except for the uniqueness, the statement is a summary of the foregoing
analysis. To prove the uniqueness, consider the diagram (\ref{eq:Representing_F})
and Proposition \ref{prop:optimality}. Since $\dot{\nabla}$ is metric
preserving and linear, $\image\dot{\nabla}$ is a Hilbert space with
the inner product inherited from $\dot{W}^{1,2}(\gO)$. Hence, there
is a unique $\ol{\vf}_{0}\in\image\dot{\nabla}\subset\dot{W}^{1,2}(\gO)$
such that
\begin{equation}
\inner{\ol w_{0}}{\gf}=\vf_{0}(\gf),\qquad\gf\in\image\dot{\nabla}.
\end{equation}
By the Riesz representation theorem 
\begin{equation}
\norm{\vf_{0}}=\sup_{\gf\in\image\dot{\nabla}}\frac{\vf_{0}(\gf)}{\norm{\gf}}=\inner{\ol{\vf}_{0}}{\ol{\vf}_{0}}^{1/2}.
\end{equation}
An obvious extension $\vf$ of $\vf_{0}$ would be in the form
\begin{equation}
\vf(\gf)=\inner{\ol{\vf}_{0}}{\gf},\qquad\gf\in\dot{W}^{1,2}(\gO).\label{eq:extension-Hilbert}
\end{equation}

Let $\ol{\vf}'\in\dot{W}^{1,2}(\gO)$ be another vector that induces
an extension of $\vf_{0}$. Then, for every $\gf\in\image\dot{\nabla}$,
\begin{equation}
\begin{split}\inner{\ol{\vf}'}{\gf} & =\inner{\ol{\vf}'-\ol{\vf}_{0}+\ol{\vf}_{0}}{\gf},\\
 & =\inner{\ol{\vf}'-\ol{\vf}_{0}}{\gf}+\inner{\ol{\vf}_{0}}{\gf},\\
 & =\inner{\ol{\vf}_{0}}{\gf}.
\end{split}
\end{equation}
Hence, $\ol{\vf}'-\ol{\vf}_{0}$ is orthogonal to $\image\dot{\nabla}$,
and 
\begin{equation}
\norm{\ol{\vf}'}^{2}=\norm{\ol{\vf}'-\ol{\vf}_{0}+\ol{\vf}_{0}}^{2}=\norm{\ol{\vf}'-\ol{\vf}_{0}}^{2}+\norm{\ol{\vf}_{0}}^{2}>\norm{\ol{\vf}_{0}}^{2}.
\end{equation}
It is concluded that the functional given by (\ref{eq:extension-Hilbert})
is the optimal extension.
\end{proof}
\begin{center}
\smallskip{}
\rule[0.5ex]{0.6\textwidth}{0.3pt}
\par\end{center}

\smallskip{}

\section{\label{sec:Capacity}Capacity of Regions}

The short analysis in this section may be motivated by considering
(the continuous counterpart of) the traffic regulation problem described
in the introduction. Assume that one has a control system so that
for every required flux density on the boundary\textemdash the traffic
leaving/entering the region\textemdash and a required rate of density
of vehicles, the flux field in the region is regulated so as to induce
an optimal flux field $\vf^{\opt}$. However, due to various constraints,
say, for example, the maximum permitted speed of the vehicles, it
is required that $\norm{\vf^{\opt}}\les M$, where $M$ is a given
bound. The question we are concerned with is what boundary flux and
density rates, can be accommodated. In other words, we are looking
for the capacity of the region and the control system to handle the
required traffic. We show that there is a minimal positive number,
$C$, depending only on the geometry of $\gO$, such that the system
can accommodate any required boundary flux and density rate if their
norms are not greater than $CM$. The capacity is simply given by
the relation $1/C=\norm{\gd}$.

In the various examples above, we have shown that optimal velocity
fields for pairs $\cons=(\wh{\vform},\wh{\sform})$, optimum flow
fields, $\vf^{\opt}$, exist and $\norm{\vf^{\opt}}=\norm{\gd^{*}(\cons)}$.
To simplify the notation and indicate the dependence on $\cons$,
we will write $\optm(\cons):=\norm{\vf^{\opt}}=\norm{\gd^{*}(\cons)}.$
(Note that we do not specify the space containing $c$ as it depends
on the particular case considered, as above.) Evidently, for any $\gl>0$,
we have
\begin{equation}
\go(\gl c)=\gl\go(c).\label{eq:homogeneity}
\end{equation}
The sensitivity of the region to the input pattern of $c$ can be
represented by the normalized optimum
\begin{equation}
K_{c}:=\frac{\go(c)}{\norm c}.\label{eq:opt_c}
\end{equation}

It is assumed now that some control system regulates the flux fields
so that for each pair $c$ as above, the regulated flux field is optimal.
Thus, the worst-case input pattern is represented by 
\begin{equation}
K=\sup_{c}K_{c}=\sup_{c}\frac{\go(c)}{\norm c}.
\end{equation}
Using (\ref{eq:opt_c}), it follows immediately that 
\begin{equation}
K=\sup_{c}\frac{\norm{\gd^{*}(\cons)}}{\norm c}=\norm{\gd^{*}}=\norm{\gd},
\end{equation}
by a standard result on dual linear operators.

It is further assumed now that the control system allows (or can handle)
only flux fields $\vf$ such that $\norm{\vf}\les M$, for a given
bound $M>0$.

Let $c'$ be any pair of input fields satisfying
\begin{equation}
\norm{c'}\les\frac{M}{K}.
\end{equation}
Then, 
\begin{equation}
\begin{split}M & \ges K\norm{c'},\\
 & =\norm{c'}\sup_{c}\left\{ \frac{\norm{\gd^{*}(\cons)}}{\norm c}\right\} ,\\
 & \ges\norm{c'}\frac{\norm{\gd^{*}(\cons')}}{\norm{c'}},\\
 & =\norm{\gd^{*}(c')}.
\end{split}
\end{equation}
It follows from (\ref{eq:opt_c}) that 
\begin{equation}
\go(c')\les M.
\end{equation}
This is summarized by

\begin{myprop}\label{prop:capacity}Let the \emph{capacity, $C$,
of the set $\gO$ }be defined as $C=1/K$. Then, the control system
can accommodate any pair of input fields, $c$, satisfying 
\begin{equation}
\norm c\les CM=\frac{M}{\norm{\gd}}.
\end{equation}

\end{myprop}

\newcommand{\etalchar}[1]{$^{#1}$}

\end{document}